\documentclass{amsart}

\usepackage[latin1]{inputenc}
\usepackage[T1]{fontenc}
\usepackage{enumerate}

\newtheorem{theorem}{Theorem}[section]
\newtheorem{lemma}[theorem]{Lemma}

\theoremstyle{definition}
\newtheorem{definition}[theorem]{Definition}

\theoremstyle{remark}
\newtheorem{remark}[theorem]{Remark}

\numberwithin{equation}{section}

\newcommand{\N}{\mathbb{N}}

\newcommand{\Z}{\mathbb{Z}}

\newcommand{\Span}{\text{span}}

\begin{document}

\title{Hereditarily hypercyclic subspaces}

\author[Q. Menet]{Quentin Menet}
\address{Institut de Mathématique\\
Université de Mons\\
20 Place du Parc\\
7000 Mons, Belgique}
\email{Quentin.Menet@umons.ac.be}
\thanks{The author is supported by a grant of FRIA}

\subjclass[2010]{Primary 47A16}
\keywords{Hypercyclic operators; Hypercyclic subspaces}

\date{}

\begin{abstract}
We say that a sequence of operators $(T_n)$ possesses hereditarily hypercyclic subspaces along a sequence $(n_k)$ if for any subsequence $(m_k)\subset(n_k)$, the sequence $(T_{m_k})$ possesses a hypercyclic subspace.  While so far no characterization of the existence of hypercyclic subspaces in the case of Fr\'{e}chet spaces is known, we succeed to obtain a characterization of sequences $(T_n)$ possessing hereditarily hypercyclic subspaces along $(n_k)$, under the assumption that the sequence $(T_n)$ satisfies the Hypercyclicity Criterion along $(n_k)$.
We also obtain a characterization of operators possessing a hypercyclic subspace under the assumption that $T$ satisfies the Frequent Hypercyclicity Criterion.
\end{abstract}
\maketitle

\section{Introduction}

We denote by $\N$ the set of positive integers and by $\Z_+$ the set of non-negative integers.

 Let $X$, $Y$ be Fréchet spaces and $L(X,Y)$ the space of continuous linear operators from $X$ to $Y$. A sequence $(T_n)_{n\ge 0}\subset L(X,Y)$ is said to be hypercyclic if there exists a vector $x\in X$ (also called a hypercyclic vector) such that the orbit $\{T_n x:n\ge 0\}$ is dense in $Y$. We denote by $\text{HC}((T_n))$ the set of hypercyclic vectors for the sequence $(T_n)$. An operator $T\in L(X)$ is said to be hypercyclic if the sequence $(T^n)$ is hypercyclic and we denote by  $\text{HC}(T)$ the set of hypercyclic vectors for $T$.
Given a hypercyclic operator $T$, one can wonder whether $T$ possesses a lot of hypercyclic vectors and which is the structure of $\text{HC}(T)$. At the beginning of the twentieth century, Birkhoff proved that if an operator $T$ is hypercyclic, then the set $\text{HC}(T)$ has to be a dense $G_{\delta}$ set. Bourdon and Herrero~\cite{Bourdon, Herrero} have then complemented this result by showing that if an operator $T$ is hypercyclic, then the set  $\text{HC}(T)\cup \{0\}$ contains a dense infinite-dimensional subspace.

We say that $(T_n)_{n\ge 0}$ possesses a hypercyclic subspace if there exists an infinite-dimensional closed subspace in which each non-zero vector is hypercyclic. Bernal and Montes~\cite{3Bernal} remarked in 1995 that non-trivial translation operators on the space of entire functions possess a hypercyclic subspace, and in 1996, Montes~\cite{3Montes} proved that the scalar multiples of the backward shift $\lambda B$ on $l^p$ ($1\le p<\infty$) do not possess any hypercyclic subspace. Therefore, a natural question is:
\begin{center}
Which hypercyclic operators possess a hypercyclic subspace?
\end{center}

In the case of \emph{complex Banach spaces}, the answer to this question was obtained by Gonz\'alez, Le\'on and Montes~\cite{3Gonzalez} in 2000 for weakly mixing operators, where an operator $T$ is said to be \emph{weakly mixing} if $T\times T$ is hypercyclic.

\begin{theorem}[\cite{3Gonzalez}]\label{3ThmMontes}
Let $X$ be a separable complex Banach space and $T\in L(X)$. If $T$ is weakly mixing, then the following assertions are equivalent:
\begin{enumerate}[\upshape (1)]
\item $T$ has a hypercyclic subspace;
\item there exist an increasing sequence $(n_k)$ of positive integers and an infinite-\\dimensional closed subspace $M_0$ of $X$ such that $T^{n_k}x\rightarrow 0$ for all $x\in M_0$;
\item there exist an increasing sequence $(n_k)$ of positive integers and an infinite-\\dimensional closed subspace $M_b$ of $X$ such that $\sup_k\|T^{n_k}|_{M_b}\|<\infty$;
\item the essential spectrum of $T$ intersects the closed unit disk.
\end{enumerate}
\end{theorem}

The key of the proof of this theorem is the notion of essential spectrum. However, this notion loses almost all its properties in the case of real Banach spaces and in the case of real or complex Fr\'{e}chet spaces. Some criteria for the existence or the non-existence of hypercyclic subspaces in the case of Fr\'{e}chet spaces have been obtained by Bernal, Bonet, Grosse-Erdmann, Mart\'{i}nez, Peris, Petersson and the author \cite{Bernal, 3Bonet, 2Grosse2, 4Menet2, 3Petersson}. Nevertheless, we do not know a complete characterization of operators with hypercyclic subspaces on these spaces.

In 1999, Bès and Peris~\cite{2Bes} showed that an operator $T$ is weakly mixing if and only if $T$ is hereditarily hypercyclic, i.e. there exists an increasing sequence $(n_k)_{k\ge 1}$ such that for any subsequence $(m_k)_{k\ge 1}\subset (n_k)_{k\ge 1}$, the sequence $(T_{m_k})_{k\ge 1}$ is hypercyclic. In view of this notion of hereditary hypercyclicity, we say that a sequence $(T_n)\subset L(X,Y)$ possesses \emph{hereditarily hypercyclic subspaces} if there exists a sequence $(n_k)$ such that for any subsequence $(m_k)\subset (n_k)$, the sequence $(T_{m_k})_{k\ge 1}$ possesses a hypercyclic subspace. The goal of this paper consists in studying which hypercyclic sequences of operators on Fréchet spaces possess hereditarily hypercyclic subspaces.

Obviously, if $T$ possesses hereditarily hypercyclic subspaces, then $T$ possesses a hypercyclic subspace and $T$ is hereditarily hypercyclic. On the other hand, it is interesting to remark that if $X$ is a complex Banach space and $T\in L(X)$, then we can deduce from Theorem~\ref{3ThmMontes} and from the equivalence between hereditarily hypercyclic operators and weakly mixing operators that $T$ possesses hereditarily hypercyclic subspaces if and only if $T$ is weakly mixing and $T$ possesses a hypercyclic subspace (Theorem~\ref{thmcomplban}). We do not know if this equivalence remains true if $T$ is an operator on real or complex Fréchet spaces. However, it motivates us to study the sequences of operators possessing hereditarily hypercyclic subspaces.

We say that a sequence $(T_n)\subset L(X,Y)$ possesses \emph{hereditarily hypercyclic subspaces along $(n_k)$} if for any subsequence $(m_k)\subset (n_k)$, the sequence $(T_{m_k})_{k\ge 1}$ possesses a hypercyclic subspace.
While so far no characterization of the existence of hypercyclic subspaces in the case of Fr\'{e}chet spaces is known, we succeed to obtain a characterization of sequences $(T_n)$ possessing hereditarily hypercyclic subspaces along $(n_k)$, under the assumption that the sequence $(T_n)$ satisfies the Hypercyclicity Criterion along $(n_k)$ (Section~\ref{6herec}). Our approach differs from the approach of Gonz\'alez, Le\'on and Montes~\cite{3Gonzalez}  by the absence of the use of spectral theory.

In Section~\ref{6FHC}, we use the ideas of the above characterization to characterize operators possessing a hypercyclic subspace under the assumption that $T$ satisfies the Frequent Hypercyclicity Criterion. If $T$ satisfies the Frequent Hypercyclicity Criterion, then $T$ is weakly-mixing. Our assumption is thus stronger than the assumption of Theorem~\ref{3ThmMontes}. In return, the obtained characterization works for any real or complex Fréchet spaces with a continuous norm and can be extended to Fréchet spaces without continuous norm through the notion of hypercyclic subspaces of type~$1$.

\section{Characterization of sequences possessing hereditarily hypercyclic subspaces}\label{6herec}

We will use the following vocabulary:

\begin{definition}
Let $X$, $Y$ be Fr\'{e}chet spaces, $(T_n)\subset L(X,Y)$ and $(n_k)_{k\ge 1}$ an increasing sequence of integers. 
\begin{itemize}
\item The sequence $(T_n)_{n\ge 0}$ possesses a \emph{hypercyclic subspace along $(n_k)$} if the sequence $(T_{n_k})_{k\ge 1}$ possesses a hypercyclic subspace.
\item The sequence $(T_n)$ possesses \emph{hereditarily hypercyclic subspaces along} $(n_k)$ if for any subsequence $(m_k)\subset(n_k)$, the sequence $(T_n)_{n\ge 0}$ possesses a hypercyclic subspace along $(m_k)$.
\item The sequence $(T_n)$ possesses \emph{hereditarily hypercyclic subspaces} if there exists a sequence $(n_k)$ such that $(T_n)$ possesses hereditarily hypercyclic subspaces along $(n_k)$. 
\end{itemize}
If $T\in L(X)$, the above definitions extend to $T$ by considering the sequence $(T^n)_{n\ge 0}$.
\end{definition}

Obviously, if $T$ possesses hereditarily hypercyclic subspaces, then the operator $T$ is hereditarily hypercyclic and possesses a hypercyclic subspace. We start by showing that if $T$ is an operator on a complex Banach space then the converse is also true. We first recall the well-known Hypercyclicity Criterion.

\begin{theorem}[Hypercyclicity Criterion \cite{2Grosse2}]
Let $X$ be a Fréchet space, $Y$ a separable Fréchet space and $(T_n)_{n\ge 0}\subset L(X,Y)$. If there are dense subsets $X_0\subset X$, $Y_0\subset Y$, an increasing sequence $(n_k)_{k\ge 1}$ of positive integers and maps $S_{n_k}:Y_0\to X$, such that for any $x\in X_0$, $y\in Y_0$,
\begin{enumerate}
 \item $T_{n_k}x\rightarrow 0$,
 \item $S_{n_k}y\rightarrow 0$,
 \item $T_{n_k}S_{n_k}y\rightarrow y$,
\end{enumerate}
then $(T_n)$ is weakly mixing and thus hypercyclic.
\end{theorem}

In fact, if $T$ is an operator on a Fréchet space, we have the following equivalences.

\begin{theorem}[Bès-Peris \cite{2Bes}]\label{2equiv wm}
Let $X$ be a separable Fréchet space and $T\in L(X)$.
The following assertions are equivalent:
\begin{itemize}
\item $T$ satisfies the Hypercyclicity Criterion;
\item $T$ is weakly mixing;
\item $T$ is hereditarily hypercyclic.
\end{itemize}
\end{theorem}
\begin{remark}
There exists some hypercyclic operators that are not weakly mixing~\cite{Rosa}.
\end{remark}

On the other hand, we have at our disposal the following criterion for the existence of hypercyclic subspaces.

\begin{theorem}[Criterion $M_0$ {\cite[Lemma 2.3]{Bernal}}]\label{3thm M0}
Let $X$ be a Fréchet space with a continuous norm, $Y$ a separable Fréchet space and $(T_n)\subset L(X,Y)$. If $(T_n)$ satisfies the Hypercyclicity Criterion along a sequence $(n_k)$ and if there exists an infinite-dimensional closed subspace $M_0$ of $X$ such that $T_{n_k}x\rightarrow 0$ for all $x\in M_0$, then $(T_n)$ possesses a hypercyclic subspace.
\end{theorem}

We can now prove the desired result. 
\begin{theorem}\label{thmcomplban}
Let $X$ be a complex Banach space and $T\in L(X)$.
The operator $T$ possesses hereditarily hypercyclic subspaces if and only if $T$ is weakly mixing and $T$ possesses a hypercyclic subspace.
\end{theorem}
\begin{proof}
If $T$ possesses hereditarily hypercyclic subspaces then $T$ is hereditarily hypercyclic and $T$ possesses a hypercyclic subspace. Therefore, since $T$ is weakly mixing if and only if $T$ is hereditarily hypercyclic (Theorem~\ref{2equiv wm}), we deduce the first implication.

On the other hand, if $T$ is weakly mixing and $T$ possesses a hypercyclic subspace, then we first deduce from Theorem~\ref{2equiv wm} that there exists an increasing sequence of integers $(n_k)$ such that $T$ satisfies the Hypercyclicity Criterion along $(n_k)$. Therefore, since $T$ possesses a hypercyclic subspace and $T$ satisfies the Hypercyclicity Criterion along $(n_k)$, we deduce from the proof of Theorem~\ref{3ThmMontes} that there exist a subsequence $(m_k)\subset (n_k)$ and an infinite-dimensional closed subspace $M_0$ such that for any $x\in M_0$, $T^{m_k}x$ converges to $0$ as $k$ tends to infinity. Finally, since $T$ satisfies the Hypercyclicity Criterion along $(n_k)$, we get that for any subsequence $(m'_k)\subset (m_k)$, the sequence $(T^{m'_k})_k$ satisfies the Hypercyclicity Criterion along the whole sequence $(k)$ and we conclude that, for any subsequence $(m'_k)\subset (m_k)$, the sequence $(T^{m'_k})$ possesses a hypercyclic subspace by using Criterion $M_0$ (Theorem~\ref{3thm M0}).
\end{proof}

Concerning Fr\'{e}chet spaces or real Banach spaces, we do not know whether every weakly mixing operator with a hypercyclic subspace possesses hereditarily hypercyclic subspaces.
However, we can obtain a characterization of sequences of operators possessing hereditarily hypercyclic subspaces along a sequence $(n_k)$. The rest of this subsection is mainly devoted to the proof of this characterization.

Before stating the characterization of sequences of operators possessing hereditarily hypercyclic subspaces.
We introduce some classic results about the notion of basic sequences.

\begin{definition}
Let $X$ be a Fréchet space. A sequence $(u_k)_{k\ge 1}\subset X$ is \emph{basic} if for every $x\in \overline{\text{span}}\{u_k:k\ge 1\}$, there exists a unique sequence $(\alpha_k)_{k\ge1}$ in $\mathbb{K}$ ($\mathbb{K}=\mathbb{R}$ or $\mathbb{C}$) such that $x=\sum_{k=1}^{\infty}\alpha_ku_k$.
\end{definition}

We can easily construct a basic sequence  in any infinite-dimensional closed subspace thanks to the following two lemmas.

\begin{lemma}[{\cite[Lemma 10.39]{2Grosse2}}]\label{3lem karl}
Let $X$ be a Fréchet space, $F$ a finite-dimensional subspace of $X$, $p$ a continuous seminorm on $X$ and $\varepsilon>0$. Then there exists a closed subspace $E$ of finite codimension such that for any $x\in E$ and $y\in F$,
\[p(x+y)\ge \max\Big(\frac{p(x)}{2+\varepsilon},\frac{p(y)}{1+\varepsilon}\Big).\]
\end{lemma}

\begin{lemma}[\cite{4Menet1}]\label{3lem basic}
Let $X$ be a Fréchet space with a continuous norm, $(p_n)$ an increasing sequence of norms defining the topology of $X$ and $(\varepsilon_n)_{n\ge 1}$ a sequence of positive numbers with $\prod_{n} (1+\varepsilon_n)=K<\infty$.
If $(u_k)_{k\ge 1}$ is a sequence of non-zero vectors in $X$ such that for any $n\ge 1$, for any $j\le n$, for any $\alpha_1,\ldots,\alpha_{n+1}\in \mathbb{K}$,
\begin{equation}p_j\Big(\sum_{k=1}^n \alpha_k u_k\Big)\le (1+\varepsilon_n)p_j\Big(\sum_{k=1}^{n+1}\alpha_k u_k\Big),\label{3lem1}
\end{equation}
then this sequence is basic in $(X,p_n)$ for any $n\ge 1$ and thus in $X$.
\end{lemma}
\begin{remark}\label{rembdd}
If $(u_k)_{k\ge 1}$ satisfies \eqref{3lem1} and $p_1(u_k)=1$ for any $k\ge 1$, then, for any convergent series $\sum_{k=1}^{\infty}\alpha_ku_k$, any $n\ge 1$,
\[|\alpha_n|=p_1(\alpha_n u_n)\le p_1 \Big(\sum_{k=1}^{n-1}\alpha_k u_k\Big) + p_1\Big(\sum_{k=1}^{n}\alpha_k u_k\Big)\le 2K p_1\Big(\sum_{k=1}^{\infty}\alpha_ku_k\Big).\]
In other words, for any convergent series $\sum_{k=1}^{\infty}\alpha_ku_k$, the sequence $(\alpha_k)_{k\ge 1}$ is bounded.
\end{remark}

%

The construction of hypercyclic subspaces will rely on the perturbation of basic sequences.

\begin{lemma}[\cite{4Menet1}]\label{3lem equiv}
Let $X$ be a Fréchet space with a continuous norm, $(p_n)_n$ an increasing sequence of norms defining the topology of $X$ and $(\varepsilon_n)_{n\ge 1}$ a sequence of positive numbers with $\prod_{n} (1+\varepsilon_n)=K<\infty$. Suppose that $(u_n)_{n\ge 1}$ is a basic sequence in $X$ satisfying \eqref{3lem1} for the sequence $(\varepsilon_n)$ and that for any $n\ge 1$, we have $p_1(u_n)=1$. If $(f_n)_{n\ge 1}\subset X$ is a sequence satisfying \[\sum_{n=1}^{\infty}2Kp_n(u_n-f_n)<1,\]
then $(f_n)_{n\ge 1}$ is a basic sequence in $X$ equivalent to $(u_n)$, i.e. for every sequence $(\alpha_k)_{k\ge 1}$ in $\mathbb{K}$, the series $\sum_{k=1}^{\infty} \alpha_k f_k$ converges in $X$ if and only if the series $\sum_{k=1}^{\infty} \alpha_k u_k$ converges in $X$.
\end{lemma}


We can now state and show the following characterization of sequences of operators possessing hereditarily hypercyclic subspaces. This characterization is the main result of this paper.

\begin{theorem}\label{6thmher}
Let $X$ be an infinite-dimensional Fr\'{e}chet space with a continuous norm, $Y$ a separable Fr\'{e}chet space and $(T_n)_{n\ge 0}\subset L(X,Y)$. If $(T_n)$ satisfies the Hypercyclicity Criterion along $(n_k)$, then the following assertions are equivalent:
\begin{enumerate}[\upshape (1)]
\item the sequence $(T_n)$ possesses hereditarily hypercyclic subspaces along $(n_k)$;
\item for any subsequence $(m_k)\subset (n_k)$, there exists an infinite-dimensional closed subspace $M\subset X$ such that for any continuous seminorm $q$ on $Y$, we have
\[\liminf_{k\to \infty} q(T_{m_k}x)<\infty \quad\text{for any }x\in M;\]
\item for any subsequence $(m_k)\subset (n_k)$, there exist an infinite-dimensional closed subspace $M_0\subset X$ and an increasing sequence of integers $k(n)$ such that for any continuous seminorm $q$ on $Y$, we have
\[\lim_{n\to \infty} \inf_{k(n-1)\le k< k(n)} q(T_{m_k}x)=0 \quad \text{for any }x\in M_0.\]
\end{enumerate}
\end{theorem}
\begin{proof}

In order to prove these equivalences, we add the following assertion:

\begin{enumerate}[\upshape (2')]
\item \emph{for any subsequence $(m_k)\subset (n_k)$, there exists an infinite-dimensional closed subspace $M\subset X$ such that for any continuous seminorm $q$ on $Y$, there exists a continuous norm $p$ on $X$ such that for any $k_0\ge 0$, there exist $k\ge k_0$ and a closed subspace $E$ of finite codimension in $M$ such that for any $x\in E$, there exists $k_0\le j\le k$ satisfying} \[q(T_{m_j}x)\le p(x).\]
\end{enumerate}

We then prove that $\neg (2')\Rightarrow \neg (2) \Rightarrow \neg (1)$ and $(2')\Rightarrow (3)\Rightarrow (1)$.\\

$\neg (2) \Rightarrow \neg (1)$: By hypothesis, there exists a subsequence $(m_k)\subset (n_k)$ such that for any infinite-dimensional closed subspace $M\subset X$,  there exist a vector $x\in M$ and a continuous seminorm $q$ on $Y$ such that
\[\lim_{k\to \infty} q(T_{m_k}x)=\infty.\]
We deduce that the sequence $(T_n)$ does not possess any hypercyclic subspace along $(m_k)$ and thus that the sequence $(T_n)$ does not possess hereditarily hypercyclic subspaces along $(n_k)$.\\

$\neg (2')\Rightarrow \neg (2)$: This implication directly follows from the following lemma.
\begin{lemma}~\label{lemher1}
Let $X$ be an infinite-dimensional Fr\'{e}chet space with a continuous norm, $Y$ a Fr\'{e}chet space, $(T_n)\subset L(X,Y)$, $(n_k)$ an increasing sequence of positive integers and $M$ an infinite-dimensional closed subspace in $X$. If there exists a continuous seminorm $q$ on $Y$ such that for any continuous norm $p$ on $X$, there exists an integer $k_0\ge 0$ such that for any $k\ge k_0$, any closed subspace $E$ of finite codimension in $M$, there exists a vector $x\in E$ satisfying for any $k_0\le j\le k$, \[q(T_{n_j}x)> p(x),\]
then there exists a vector $x\in M$ such that \[\lim_{k\to \infty} q(T_{n_k}x)=\infty.\]
\end{lemma}
\begin{proof}
Let $q$ be a continuous seminorm on $Y$ such that for any continuous norm $p$ on $X$, there exists an integer $k_0\ge 0$ such that for any $k\ge k_0$, any closed subspace $E$ of finite codimension in $M$, there exists a vector $x\in E$ satisfying for any $k_0\le j\le k$, \[q(T_{n_j}x)> p(x).\]
We seek to construct a vector $x\in M$ such that \[\lim_{k\to \infty} q(T_{n_k}x)=\infty.\]
The idea of the construction of this vector $x$ is similar to the construction given in \cite[Theorem 1.13]{4Menet2}.
(Our hypotheses are in fact the minimal conditions so that the proof of~\cite[Theorem 1.13]{4Menet2} works.)

Let $(p_n)_{n\ge 1}$ be an increasing sequence of norms inducing the topology of $X$. For each continuous norm $n^3p_n$, there exists an integer $k_n\ge 1$ such that for any $k\ge k_n$, any closed subspace $E$ of finite codimension in $M$, there exists a vector $x\in E$ satisfying for any $k_n\le j\le k$, \[q(T_{n_j}x)> n^3p_n(x).\] 
Without loss of generality, we can assume that the sequence $(k_n)_{n\ge 1}$ is increasing.
We can then construct recursively a sequence $(e_n)_{n\ge 0}\subset M$ with $e_0=0$ such that for any $n\ge 1$, we have
\begin{itemize}
\item $p_n(e_n)=\frac{1}{n^2}$;
\item for any $k_{n}\le j<k_{n+1}$, $q(T_{n_j} e_n)> n^3p_n(e_n)=n$;
\item for any $j< k_{n+1}$, $T_{n_j}e_n\in \bigcap_{k\le n-1} L_{k,j}$, where $L_{k,j}$ is the closed subspace of finite codimension given by Lemma \ref{3lem karl} for $F_{k,j}=\text{span}(T_{n_j} e_0, \dots, T_{n_j}e_k)$, $q$ and $\varepsilon>0$.
\end{itemize}
We then let $x:=\sum_{n=1}^{\infty} e_n$ so that $x\in M$ and for any $n\ge 1$, any $k_{n}\le j< k_{n+1}$,
\begin{align*}
q(T_{n_j}x)=q\Big(\sum_{\nu=1}^{\infty}T_{n_j}e_{\nu}\Big)&\ge \frac{1}{1+\varepsilon}q\Big(\sum_{\nu=1}^{n}T_{n_j}e_{\nu}\Big) \quad\text{because $\sum_{\nu=n+1}^{\infty}T_{n_j}e_{\nu}\in L_{n,j}$}\\
&\ge \frac{1}{(1+\varepsilon)(2+\varepsilon)}q(T_{n_j}e_n)\quad\text{because $T_{n_j}e_n\in L_{n-1,j}$}\\
&\ge \frac{n}{(1+\varepsilon)(2+\varepsilon)}.
\end{align*}
\end{proof}

$(2')\Rightarrow (3)$:   Since $(T_n)$ satisfies the Hypercyclicity Criterion along $(n_k)$, the sequence $(T_n)$ satisfies the Hypercyclicity Criterion along any subsequence of $(n_k)$. We can then deduce the implication $(2')\Rightarrow (3)$ from the following lemma.

\begin{lemma}\label{lemher2}
Let $X$ be an infinite-dimensional Fr\'{e}chet space with a continuous norm, $Y$ a separable Fr\'{e}chet space and $T_n: X\rightarrow Y$ a sequence of operators satisfying the Hypercyclicity Criterion along $(n_k)$. Let $M$ be an infinite-dimensional closed subspace of $X$. If for any continuous seminorm $q$ on $Y$, there exists a continuous norm $p$ on $X$ such that for any $k_0\ge 0$, there exist $k\ge k_0$ and a closed subspace $E$ of finite codimension in $M$ such that for any $x\in E$, there exists $k_0\le j\le k$ satisfying \[q(T_{n_j}x)\le p(x),\]
then there exist an infinite-dimensional closed subspace $M_0\subset X$ and an increasing sequence of integers $k(n)$ such that for any continuous seminorm $q$ on $Y$, we have
\[\lim_{n\to \infty} \inf_{k(n-1)\le k< k(n)} q(T_{n_k}x)=0 \quad \text{for any }x\in M_0.\]
\end{lemma}
\begin{proof}
Let $(q_n)$ be an increasing sequence of seminorms inducing the topology of $Y$. There exists by hypothesis a sequence $(p_n)$ of norms on $X$ such that for any $n\ge 1$, any $k_0\ge 0$, there exist $k\ge k_0$ and a closed subspace $E$ of finite codimension in $M$ such that for any $x\in E$, there exists $k_0\le l\le k$ satisfying
 \[q_n(T_{n_l}x)\le p_n(x).\]
Without loss of generality, we can assume that the sequence $(p_n)$ is an increasing sequence inducing the topology of $X$.
 
Let $X_0$ be a dense subset of $X$ given by the Hypercyclicity Criterion for $(n_k)$ and $(\varepsilon_n)_{n\ge 1}$ a sequence of positive numbers such that ${\prod_{n=1}^{\infty}(1+\varepsilon_n)\le 2}$. We construct recursively a sequence $(E_n)_n$ of closed subspaces of finite codimension in $M$, a basic sequence $(u_n)_{n\ge 1}\subset M$, a basic sequence $(f_n)_{n\ge 1}\subset X_0$ equivalent to $(u_n)_{n\ge 1}$, an increasing sequence of positive integers $(k(n))_{n\ge 1}$  with $k(1)=1$ and a sequence of integers $(l_n)$ such that for any $n\ge 1$,
\begin{enumerate}[\upshape (i)]
\item for any $x\in E_n$, any $k\le n$, there exists $k(n)\le l\le k(n)+ l_n$ such that\[q_k(T_{n_{l}}x)\le p_k(x);\]
\item $u_n\in \bigcap_{k\le n}E_k$ and $p_1(u_n)=1$;
\item for any $m\le n-1$, any $\alpha_1,\ldots,\alpha_{n}\in \mathbb{K}$,
\[p_m\Big(\sum_{k=1}^{n-1} \alpha_k u_k\Big)\le (1+\varepsilon_{n-1})p_m\Big(\sum_{k=1}^{n}\alpha_k u_k\Big)\quad \quad (n\ge 2).\]
\item $p_n(u_n-f_n)<\frac{1}{2^{n+2}}$;
\item for any $m\le n$, any $ k(m)\le l\le k(m)+l_m$, \[q_m(T_{n_{l}}(f_n-u_n))\le \frac{1}{2^{m+n}};\]
\item $k(n+1)> k(n)+l_n$;
\item for any $m\le n$, any $l\ge k(n+1)$, \[q_{n+1}(T_{n_{l}} f_m)\le \frac{1}{2^{m+n+1}}.\]
\end{enumerate}
We start by explaining the base case of this construction by induction. We know by hypothesis that there exist an integer $l_1\ge 0$ and a closed subspace $E_1$ of finite codimension in $M$ such that for any $x\in E_1$, there exists $k(1)\le l\le k(1)+l_1$ satisfying \[q_1(T_{n_l}x)\le p_1(x).\]
We choose $u_1\in E_1$ such that $p_1(u_1)=1$ and we choose $f_{1}\in X_0$ close enough to $u_1$ so that (iv) and (v) are satisfied. Finally, we choose $k(2)> k(1)+l_1$ sufficiently large so that (vii) is satisfied.

Suppose that we have already constructed $E_1,\dots, E_{n-1}$, $u_1,\dots,u_{n-1}$, $f_1,\dots,f_{n-1}$, $l_1,\dots,l_{n-1}$ and $k(1),\dots,k(n)$. By hypothesis, for any $m\le n$, there exists $l_{n,m}\ge 0$ and a closed subspace $E_{n,m}$ of finite codimension in $M$ such that for any $x\in E_{n,m}$, there exists $k(n)\le l\le k(n)+ l_{n,m}$ satisfying \[q_{m}(T_{n_{l}}x)\le p_{m}(x).\]
We let $E_{n}=\bigcap_{m=1}^{n}E_{n,m}$ and $l_{n}=\max_{1\le m\le n} l_{n,m}$. We select $u_{n}\in \bigcap_{k\le n}E_k$ such that $p_1(u_{n})=1$ and 
for any $m\le n-1$, any $\alpha_1,\ldots,\alpha_{n}\in \mathbb{K}$,
\[p_m\Big(\sum_{k=1}^{n-1} \alpha_k u_k\Big)\le (1+\varepsilon_n)p_m\Big(\sum_{k=1}^{n}\alpha_k u_k\Big).\]
Such a choice is possible thanks to Lemma~\ref{3lem karl}. We then choose $f_{n}\in X_0$ close enough to $u_{n}$ so that (iv) and (v) are satisfied. We finish by fixing $k(n+1)> k(n)+l_n$ sufficiently large so that (vii) is satisfied.

By Lemma~\ref{3lem basic} and Lemma~\ref{3lem equiv}, the sequences $(u_n)_{n\ge 1}$ and $(f_n)_{n\ge 1}$ are equivalent basic sequences. We let $M_0:=\overline{\Span}\{f_n:n\ge1\}$ and we show that for any $x\in M_0$, any $k\ge 1$,
\begin{equation}\lim_{n\to \infty} \inf_{k(n)\le l\le k(n)+l_n} q_k(T_{n_{l}}x)=0.\label{6eqliminf}
\end{equation}
This will conclude the proof because for any $n\ge 1$, $[k(n),k(n)+l_n]\subset[k(n),k(n+1)[$ and $(q_k)$ is an increasing sequence of seminorms inducing the topology of $Y$.

Let $x=\sum_{n=1}^{\infty}\alpha_nf_n\in M_0$. Since for any $n\ge 1$, $p_1(u_n)=1$ and $(u_n)_{n\ge 1}$ is equivalent to $(f_n)_{n\ge 1}$, we know that the sequence $(\alpha_n)_{n\ge 1}$ is bounded by some constant $K$ (Remark~\ref{rembdd}). Let $k\ge 1$. For any $n\ge k$, any $k(n)\le l\le k(n)+l_n$, we deduce from (v) and (vii) that
\begin{align*}
q_k(T_{n_{l}}x)&= q_k\Big(T_{n_{l}}\Big(\sum_{m=1}^{\infty} \alpha_m f_m\Big)\Big)\\
&\le \sum_{m<n} |\alpha_m| q_n(T_{n_{l}} f_m) + q_k\Big(T_{n_{l}}\Big(\sum_{m\ge n}\alpha_m u_m\Big)\Big)\\
&\quad+ \sum_{m\ge n} |\alpha_{m}| q_n(T_{n_{l}}(f_m-u_m))\\
&\le \sum_{m<n} K \frac{1}{2^{n+m}} + q_k\Big(T_{n_{l}}\Big(\sum_{m\ge n} \alpha_m u_m\Big)\Big)
+ \sum_{m\ge n} K \frac{1}{2^{n+m}}\\
&\le \frac{K}{2^{n}} + q_k\Big(T_{n_{l}}\Big(\sum_{m\ge n} \alpha_m u_m\Big)\Big)
\end{align*}
and as $\sum_{m\ge n} \alpha_m u_m\in E_n$, we deduce from (i) that
\begin{align*}
\inf_{k(n)\le l\le k(n)+l_n} q_k(T_{n_{l}}x)&\le \frac{K}{2^{n}} + \inf_{k(n)\le l\le k(n)+l_n}q_k\Big(T_{n_{l}}\Big(\sum_{m\ge n} \alpha_m u_m\Big)\Big)\\
&\le \frac{K}{2^{n}} + p_k\Big(\sum_{m\ge n} \alpha_m u_m\Big) \xrightarrow[n\to \infty]{} 0. 
\end{align*}
We conclude that \eqref{6eqliminf} is satisfied.
\end{proof}

$(3)\Rightarrow (1)$: We remark that it suffices to prove that if $(T_n)$ satisfies the Hypercyclicity Criterion along $(n_k)$ and if for any subsequence $(m_k)\subset (n_k)$, there exist an infinite-dimensional closed subspace $M_0\subset X$ and an increasing sequence of integers $k(n)$ such that for any continuous seminorm $q$ on $Y$, we have
\[\lim_{n\to \infty} \inf_{k(n-1)\le k< k(n)} q(T_{m_k}x)=0 \quad \text{for any }x\in M_0,\]
then $(T_n)$ possesses a hypercyclic subspace along $(n_k)$.

To this end, we will work with a special subsequence $(m_k)\subset (n_k)$ given by the following lemma.

\begin{lemma}\label{6lemher 3}
Let $X$ be an infinite-dimensional Fr\'{e}chet space with a continuous norm, $Y$ a separable Fr\'{e}chet space and $(T_n)\subset L(X,Y)$. Let $(p_n)$ be an increasing sequence of norms inducing the topology of $X$ and $(q_n)$ an increasing sequence of seminorms inducing the topology of $Y$. Let $(y_i)_{i\ge 1}$ be a dense sequence in $Y$.
If $(T_n)$ satisfies the Hypercyclicity Criterion along $(n_k)$, then there exists a subsequence $(m_k)\subset (n_k)$ such that for any $i\le k<j$, there exists $x\in X$ such that
\begin{enumerate}[\upshape 1.]
\item $p_k(x)< \frac{1}{2^k}$;
\item for any $l\in [k,j[$, $q_k(T_{m_l}x-y_{i})<\frac{1}{2^{k+1}}$;
\item for any $l\notin [k,j[$, $q_{l}(T_{m_l}x)<\frac{1}{2^{k+l+1}}$.
\end{enumerate}
\end{lemma}
\begin{proof}
Let $X_0$ (resp. $Y_0$) be the dense subset in $X$ (resp. $Y$) and $(S_{n_k})$ the sequence of maps given by the Hypercyclicity Criterion along $(n_k)$. We start by proving the existence of a subsequence $(m_l)\subset (n_k)$ and a family $(x_{i,l})_{i\le l}\subset X_0$ such that for any $i\le l$, 
\begin{enumerate}[\upshape (i)]
\item $p_l(x_{i,l})< \frac{1}{2^{l+1}}$;
\item $q_l(T_{m_l}x_{i,l}-y_{i})<\frac{1}{2^{l+2}}$;
\item for any $k< l$, $q_{k}(T_{m_{k}}x_{i,l})<\frac{1}{2^{k+l+2}}$;
\item for any $j\le k<l$, $q_l(T_{m_l}x_{j,k})<\frac{1}{2^{k+l+2}}$.
\end{enumerate}
Suppose that $m_{k}$ and $x_{j,k}\in X_0$ have already been chosen for any $j\le k\le l-1$. We first remark that if we choose $m_l\in (n_k)_{k\ge 1}$ sufficiently large, then for any $j\le k< l$,
\[q_l(T_{m_l}x_{j,k})<\frac{1}{2^{k+l+2}}.\]
 On the other hand, if $m_l\in (n_k)_{k\ge 1}$ is sufficiently large, then for any $i\le l$ we can find a vector $x_{i,l}\in X_0$ satisfying (i), (ii) and (iii).
Indeed, if we consider $y'_i\in Y_0$ such that $q_l(y'_i-y_i)<\frac{1}{2^{l+3}}$, there exists an integer $m$ such that if $m_l\ge m$ with $m_l\in (n_k)_{k\ge 1}$, then for any $k<l$, \[p_l(S_{m_l}y'_i)< \frac{1}{2^{l+1}},\ q_l(T_{m_l}S_{m_l}y'_i-y'_{i})<\frac{1}{2^{l+3}}\ \text{and}\ q_{k}(T_{m_{k}}S_{m_l}y'_i)<\frac{1}{2^{k+l+2}}.\] It then suffices to consider $x_{i,l}\in X_0$ sufficiently close to $S_{m_l}y'_i$.

We remark that it follows from (iii) and (iv) that for any $i\le l'$, any $l\ne l'$,
\begin{equation}
q_{l}(T_{m_{l}}x_{i,l'})<\frac{1}{2^{l+l'+2}}.
\label{6(vi)}
\end{equation}
We now prove that for any $i\le k<j$, there exists $x\in X$ such that 1., 2. and 3. hold. Let $i\le k<j$ and $x:=\sum_{l=k}^{j-1}x_{i,l}$. We have
\begin{enumerate}[\upshape 1.]
\item $p_k(x)\le \sum_{l=k}^{j-1}p_l(x_{i,l})< \sum_{l=k}^{j-1} \frac{1}{2^{l+1}}< \frac{1}{2^k}$ \quad \text{by (i)};
\item for any $l\in [k,j[$, 
\begin{align*}
q_k(T_{m_l}x-y_{i})&\le q_l(T_{m_l}x_{i,l}-y_{i})+ \sum_{l'\in [k,j[\backslash\{l\}}q_{l}(T_{m_l}x_{i,l'})\\
&<\frac{1}{2^{l+2}}+ \sum_{l'\in [k,j[\backslash\{l\}} \frac{1}{2^{l+l'+2}} < \frac{1}{2^{k+1}} \quad \text{by (ii) and \eqref{6(vi)}};
\end{align*}
\item for any $l\notin [k,j[$, 
\[
q_{l}(T_{m_l}x)\le \sum_{l'=k}^{j-1}q_{l}(T_{m_{l}}x_{i,l'})
 <\sum_{l'=k}^{j-1} \frac{1}{2^{l+l'+2}}< \frac{1}{2^{k+l+1}} \quad \text{by \eqref{6(vi)}}.
\]
\end{enumerate}
\end{proof}

Let $(p_n)$ be an increasing sequence of norms inducing the topology of $X$ and $(q_n)$ an increasing sequence of seminorms inducing the topology of $Y$. Let $(y_i)_{i\ge 1}$ be a dense sequence in $Y$ and $(m_k)$ the subsequence given by Lemma~\ref{6lemher 3}. By hypothesis, there then exists an infinite-dimensional closed subspace $M_0\subset X$ and an increasing sequence $(k(n))_n$ such that for any continuous seminorm $q$ on $Y$, we have
\[\lim_{n\to \infty} \inf_{k(n-1)\le k< k(n)} q(T_{m_k}x)=0 \quad \text{for any $x\in M_0$}.\]
Let $(u_n)$ be a basic sequence in $M_0$ given by Lemma~\ref{3lem basic} for $K=2$ such that $p_1(u_n)=1$ for any $n\ge 1$. Let $\phi$ be an injective function from $\N\times\N$ to $\N$ such that for any $n,j\ge 1$, we have $\phi(n,j)\ge n+j+2$ and such that for any $n\ge 1$, $(\phi(n,j))_j$ is increasing. We remark that for any $(n,j)\ne (m,i)$, the intervals $[k(\phi(n,j)),k(\phi(n,j)+1)[$ and $[k(\phi(m,i)),k(\phi(m,i)+1)[$ are disjoint. Thanks to the properties of the sequence $(m_k)$, we can then select, for any $n,j\ge 1$, a vector $x_{n,j}\in X$ such that
\begin{enumerate}
\item $p_{k(\phi(n,j))}(x_{n,j})< \frac{1}{2^{k(\phi(n,j))}}$;
\item for any $l\in [k(\phi(n,j)),k(\phi(n,j)+1)[$, \[q_{k(\phi(n,j))}(T_{m_l}x_{n,j}-y_{j})<\frac{1}{2^{k(\phi(n,j))+1}};\]
\item for any $l\notin [k(\phi(n,j)),k(\phi(n,j)+1)[$, \[q_{l}(T_{m_l}x_{n,j})<\frac{1}{2^{k(\phi(n,j))+l+1}}.\]
\end{enumerate}
 
For any $n,j\ge 1$, since $\phi(n,j)\ge n+j+2$ and $k(\cdot)$ is increasing, we know that $k(\phi(n,j))\ge n+j+2$ and thus by definition of $x_{n,j}$ that
\[p_{n+j}(x_{n,j})<\frac{1}{2^{n+j+2}}.\]
We deduce that each element $f_n:=u_n +\sum_{j=1}^{\infty} x_{n,j}$ is well-defined. Moreover, the sequence $(f_n)$ is a basic sequence equivalent to $(u_n)$ because
\[
\sum_{n\ge 1}4p_n(f_n-u_n)\le \sum_{n\ge 1}\sum_{j\ge 1}4p_n(x_{n,j})< \sum_{n\ge 1}\sum_{j\ge 1} 4\frac{1}{2^{n+j+2}}=1  \quad (\text{Lemma~\ref{3lem equiv}}).\]
We let $M_f=\overline{\Span}\{f_n:n\ge 1\}$ and we show that $M_f$ is a hypercyclic subspace along $(m_k)$ and thus a hypercyclic subspace along $(n_k)$. 

We first show that for any $i,n\ge 1$, any $k(\phi(n,i))\le l< k(\phi(n,i)+1)$
\begin{equation}
q_i(T_{m_l}(f_n-u_n)-y_i)\le \frac{1}{2^i}.\label{6hered1}
\end{equation}
Indeed, for any $i,n\ge 1$, any $k(\phi(n,i))\le l< k(\phi(n,i)+1)$, we have
\begin{align*}
q_i(T_{m_l}(f_n-u_n)-y_i)&\le \sum_{j\ne i}q_{k(\phi(n,i))}(T_{m_l}x_{n,j})
+ q_{k(\phi(n,i))}(T_{m_l}x_{n,i}-y_i)\\
&\le \sum_{j\ne i} q_{l}(T_{m_l}x_{n,j})
+ q_{k(\phi(n,i))}(T_{m_l}x_{n,i}-y_i)\\
&< \sum_{j\ne i} \frac{1}{2^{k(\phi(n,j))+l+1}} + \frac{1}{2^{k(\phi(n,i))+1}}\\
&\le \frac{1}{2^{l+n}}+ \frac{1}{2^{n+i+1}} \le \frac{1}{2^{i}}.
\end{align*}
We also remark that for any $i\ge 1$, for any $n\ne m$, for any $k(\phi(n,i))\le l< k(\phi(n,i)+1)$ 
\begin{equation}
q_i(T_{m_l}(f_m-u_m))< \frac{1}{2^{i+m}}
\label{6hered2bis}
\end{equation}
because
\begin{align*}
q_i(T_{m_l}(f_m-u_m))&\le \sum_{j\ge 1} q_{k(\phi(n,i))}(T_{m_l}x_{m,j})\\
&\le \sum_{j\ge 1} q_{l}(T_{m_l}x_{m,j}) <  \sum_{j\ge 1} \frac{1}{2^{k(\phi(m,j))+l+1}}\\
&\le \frac{1}{2^{l+m}}\le \frac{1}{2^{i+m}}.
\end{align*} 

Let $x=\sum_{m\ge1}\alpha_m f_m\in M_f\backslash\{0\}$. There exists $n\ge 1$ such that $\alpha_n\ne 0$ and without loss of generality, we can suppose that $\alpha_n=1$. Moreover, we know that the sequence $(\alpha_m)_m$ is bounded by some constant $K$ (Remark~\ref{rembdd}). Therefore, for any $k\le i$, any $k(\phi(n,i))\le l< k(\phi(n,i)+1)$,
\begin{align*}
&q_k\Big(T_{m_l}\Big(\sum_{m\ge 1} \alpha_m f_m\Big)-y_i\Big)\\
&\quad\le q_k\Big(T_{m_l}\Big(\sum_{m\ge 1} \alpha_m u_m\Big)\Big)
+ q_k(T_{m_l}(f_{n}-u_{n})-y_i)
+ \sum_{m\ne n} K q_k(T_{m_l}(f_m-u_m))\\
&\quad\le  q_k\Big(T_{m_l}\Big(\sum_m \alpha_m u_m\Big)\Big)
+ q_i(T_{m_l}(f_{n}-u_{n})-y_i)
+ \sum_{m\ne n} K q_i(T_{m_l}(f_m-u_m))\\
&\quad\le q_k\Big(T_{m_l}\Big(\sum_m \alpha_m u_m\Big)\Big)+ \frac{1}{2^i}+ \sum_{m\ne n} \frac{K}{2^{i+m}}\quad \text{by \eqref{6hered1} and \eqref{6hered2bis}}\\
&\quad\le q_k\Big(T_{m_l}\Big(\sum_m \alpha_m u_m\Big)\Big)+ \frac{1}{2^i}+ \frac{K}{2^{i}}.
\end{align*}
Since $\sum_m \alpha_m u_m\in M_0$, we know that for any $k\ge 1$, there exists an increasing sequence $(l_i)_{i\ge 1}$ such that $k(\phi(n,i))\le l_i< k(\phi(n,i)+1)$ and $q_k\big(T_{m_{l_i}}\big(\sum_m \alpha_m u_m\big)\big)$ tends to $0$ as $i$ tends to infinity. Therefore, we deduce that $q_k\Big(T_{m_{l_i}}x-y_i\Big)$ converges to $0$ as $i$ tends to $\infty$ and thus that $x$ is hypercyclic along the sequence $(m_k)$.
\end{proof}

In the case of Fr\'{e}chet spaces without continuous norm, Theorem~\ref{6thmher} and its proof can easily be generalized to hypercyclic subspaces of type $1$. A hypercyclic subspace $M\subset X$ is said to be of type $1$ if
there exists a continuous seminorm $p$ on $X$ such that $M\cap \ker p=\{0\}$. This notion has been introduced in \cite{4Menet3}.
\begin{definition}
Let $X$, $Y$ be Fr\'{e}chet spaces, $(T_n)\subset L(X,Y)$ and $(n_k)_{k\ge 1}$ an increasing sequence of integers. The sequence $(T_n)$ possesses \emph{hereditarily hypercyclic subspaces of type} $1$ along $(n_k)$ if for any subsequence $(m_k)\subset(n_k)$, the sequence $(T_{m_k})_{k\ge 1}$ possesses a hypercyclic subspace of type~$1$.
\end{definition}

\begin{theorem}
Let $X$ be an infinite-dimensional Fr\'{e}chet space, $Y$ a separable Fr\'{e}chet space and $(T_n)\subset L(X,Y)$.
If $(T_n)$ satisfies the Hypercyclicity Criterion along $(n_k)$, then the following assertions are equivalent:
\begin{enumerate}[\upshape (1)]
\item the sequence $(T_n)$ possesses hereditarily hypercyclic subspaces of type $1$ along~$(n_k)$;
\item for any subsequence $(m_k)\subset (n_k)$, there exist an infinite-dimensional closed subspace $M\subset X$ and a continuous seminorm $p$ on $X$ such that 
\begin{itemize}
\item $M\cap \ker p$ is of infinite codimension in $M$,
\item for any continuous seminorm $q$ on $Y$, we have
\[\liminf_{k\to \infty} q(T_{m_k}x)<\infty \quad\text{for any }x\in M;\]
\end{itemize}
\item for any subsequence $(m_k)\subset (n_k)$, there exist an infinite-dimensional closed subspace $M_0\subset X$, a continuous seminorm $p$ on $X$ and an increasing sequence of integers $k(n)$ such that 
\begin{itemize}
\item $M_0\cap \ker p$ is of infinite codimension in $M_0$,
\item for any continuous seminorm $q$ on $Y$, we have
\[\lim_{n\to\infty} \inf_{k(n-1)\le k< k(n)} q(T_{m_k}x)=0 \quad \text{for any }x\in M_0.\]
\end{itemize}
\end{enumerate}
\end{theorem}

\section{Characterization of the existence of hypercyclic subspaces for operators satisfying the Frequent Hypercyclicity Criterion}\label{6FHC}

By using the ideas of Theorem~\ref{6thmher}, we can obtain a characterization of operators~$T$ with a hypercyclic subspace under the assumption that $T$ satisfies the Frequent Hypercyclicity Criterion.

\begin{definition}[Frequent Hypercyclicity Criterion, \cite{4Bonilla}]
Let $X$ be a Fr\'{e}chet space, $Y$ a separable Fr\'{e}chet space and $(T_n)_{n\ge 0}\subset L(X,Y)$. We say that $(T_n)$ satisfies the Frequent Hypercyclicity Criterion if there are a dense subset $Y_0$ of $Y$ and maps $S_n:Y_0\rightarrow X$ such that, for each $y\in Y_0$,
\begin{enumerate}[\upshape 1.]
\item $\sum_{n=0}^{k}T_{k}S_{k-n}y$ converges unconditionally in $Y$, uniformly in $k\in \Z_+$,
\item $\sum_{n=0}^{\infty}T_{k}S_{k+n}y$ converges unconditionally in $Y$, uniformly in $k\in \Z_+$,
\item $\sum_{n=0}^{\infty}S_ny$ converges unconditionally in $X$,
\item $T_nS_ny\rightarrow y$.
\end{enumerate}
\end{definition}
\begin{remark}
A collection of series $\sum_{n=0}^{\infty}x_{n,k}$, $k\in I$, is said to be unconditionally convergent uniformly in $k\in I$ if for any $\varepsilon>0$, there exists $N\ge 0$ such that for any $k\in I$, any finite set $F\subset [N,\infty[$,
\[\Big\|\sum_{n\in F}x_{k,n}\Big\|<\varepsilon.\]
\end{remark}
\begin{remark}
If $T\in L(X)$ satisfies the Frequent Hypercyclicity Criterion, then $T$ satisfies the Hypercyclicity Criterion along the whole sequence $(n)$.
\end{remark}

The characterization that we obtain can be stated as follows.


\begin{theorem}\label{6thmherfhc}
Let $X$ be an infinite-dimensional separable Fr\'{e}chet space with a continuous norm and $T\in L(X)$.
If $T$ satisfies the Frequent Hypercyclicity Criterion, then the following assertions are equivalent:
\begin{enumerate}[\upshape (1)]
\item the operator $T$ possesses a hypercyclic subspace;
\item there exists an infinite-dimensional closed subspace $M\subset X$ such that for any continuous seminorm $p$ on $X$, we have
\[\liminf_{k\to \infty} p(T^kx)<\infty \quad\text{for any }x\in M;\]
\item there exist an infinite-dimensional closed subspace $M_0\subset X$ and an increasing sequence of integers $k(n)$ such that for any continuous seminorm $p$ on $X$, we have
\[\lim_{n \to \infty} \inf_{k(n-1)\le k< k(n)} p(T^kx)=0 \quad \text{for any }x\in M_0.\]
\end{enumerate}
\end{theorem}
\begin{proof}
As for the proof of Theorem~\ref{6thmher}, we consider the following intermediate assertion:
\begin{enumerate}[\upshape (2')]
\item \emph{there exists an infinite-dimensional closed subspace $M\subset X$ such that for any continuous seminorm $q$ on $X$, there exists a continuous norm $p$ on $X$ such that for any $k_0\ge 0$, there exist $k\ge k_0$ and a closed subspace $E$ of finite codimension in $M$ such that for any $x\in E$, there exists $k_0\le j\le k$ satisfying} \[q(T^jx)\le p(x).\]
\end{enumerate}

We know that $\neg (2)\Rightarrow \neg(1)$ and in view of Lemma~\ref{lemher1} and Lemma~\ref{lemher2}, we deduce that
$\neg(2')\Rightarrow \neg(2)$ and $(2')\Rightarrow (3)$. It remains to prove that $(3)\Rightarrow (1)$.

$(3)\Rightarrow (1)$. Let $(p_n)$ be an increasing sequence of norms inducing the topology of $X$. Let $Y_0\subset X$ be the dense subset and $S_n$ the maps given by the Frequent Hypercyclicity Criterion. Let $(y_j)_{j\ge 1}$ be a dense sequence in $Y_0$. 

By hypothesis, there exist an infinite-dimensional closed subspace $M_0$ and an increasing sequence $(k(n))_n$ such that for any continuous seminorm $p$ on $X$, we have
\[\lim_{n\to \infty} \inf_{k(n-1)\le k< k(n)} p(T^kx)=0 \quad \text{for any $x\in M_0$}.\]
Let $(u_n)$ be a basic sequence in $M_0$ given by Lemma~\ref{3lem basic} for $K=2$ such that $p_1(u_n)=1$ for any $n\ge 1$. Let $\psi$ be a function from $\N$ to $\N\times\N\times \N$ such that for any $n,j,N\ge 1$, $\#\{l\in\N:\psi(l)=(n,j,N)\}=\infty$. We construct an increasing sequence $(n_l)_{l\ge 1}$ such that for any $l\ge 1$, if $\psi(l)=(n,j,N)$, then 
\begin{enumerate}[\upshape (i)]
\item $k(n_l+1)+N\le k(n_{l+1})$,
\item for any $k\ge k(n_l)$, $p_{l+j}(T^kS_ky_j-y_j)<\frac{1}{2^{l+j+1}}$
\end{enumerate}
and the vector
$x_{l}:=\sum_{i=0}^{K}S_{k(n_l)+iN}y_j$, where \[K=\max\{i\ge 0:k(n_l)+iN\in[k(n_l),k(n_l+1)+N[\},\] satisfies
\begin{enumerate}[\upshape (i)]
\setcounter{enumi}{2}
\item $p_{l+n}(x_l)<\frac{1}{2^{l+2}}$,
\item for any $l'\ne l$, if $\psi(l')=(n',j',N')$, then for any $k\in [k(n_{l'}),k(n_{l'}+1)+N'[$, 
\[p_{l+j'}(T^kx_l)<\frac{1}{2^{l+l'+n'+j'}}.\]
\end{enumerate}
If $\psi(1)=(n,j,N)$, we start by choosing $n_1$ sufficiently large so that $p_{1+n}(x_1)<1/2^{3}$ and for any $k\ge k(n_1)$, $p_{1+j}(T^kS_ky_j-y_j)<\frac{1}{2^{2+j}}$. This is possible because $\sum_{n=0}^{\infty}S_ny_j$ converges unconditionally in $X$ and $T^kS_ky_j$ tends to $y_j$ as $k$ tends to infinity. If we now suppose that $n_1,\dots,n_l$ are fixed, we can choose $n_{l+1}$ sufficiently large so that (i), (ii) and (iii) are satisfied and so that for any $l'\le l$, if
 $\psi(l')=(n',j',N')$, then for any $k\in [k(n_{l'}),k(n_{l'}+1)+N'[$, 
\begin{equation}
p_{l+1+j'}(T^kx_{l+1})<\frac{1}{2^{l+1+l'+n'+j'}}
\label{6xfhcpet}
\end{equation}
and if $\psi(l+1)=(n,j,N)$, then for any $k\in [k(n_{l+1}),k(n_{l+1}+1)+N[$, 
\begin{equation}
p_{l'+j}(T^kx_{l'})<\frac{1}{2^{l'+l+1+n+j}}.
\label{6fhxcond1}
\end{equation}
Indeed, if $n_{l+1}$ is sufficiently large, we know that (ii) is satisfied because $T^kS_ky_j$ tends to $y_j$ as $k$ tends to infinity, and that inequalities (iii) and \eqref{6xfhcpet} are satisfied because the series $\sum_{n=0}^{\infty}S_ny_j$ converges unconditionally and thus $x_l$ is as small as desired. Finally, \eqref{6fhxcond1} is satisfied if $n_{l+1}$ is sufficiently large because for any $k\ge 0$, any $j'\ge 1$,
the series $\sum_{n=0}^{k}T^{k}S_{k-n}y_{j'}$ converges unconditionally in $X$. Indeed, for any $l'\le l$, any $k\ge k(n_{l+1})$, if $\psi(l')=(n',j',N')$, we have
\[T^kx_{l'}=\sum_{i\in F}T^kS_{k-i}y_{j'}\quad\text{with}\quad F\subset [k(n_{l+1})-k(n_{l'})-K'N',\infty[.\]

%

 
Let $A_n=\{l\ge 1: \psi(l)=(n,*,*)\}$. We let $f_n:=u_n +\sum_{l\in A_n}x_{l}$, which clearly converges thanks to (iii). Moreover, the sequence $(f_n)$ is a basic sequence equivalent to $(u_n)$ because
\[
\sum_{n=1}^{\infty}4p_n(f_n-u_n)\le \sum_{n=1}^{\infty}\sum_{l\in A_n}4p_n(x_{l})< \sum_{n=1}^{\infty}\sum_{l\in A_n} \frac{4}{2^{l+2}}= \sum_{l=1}^{\infty} \frac{4}{2^{l+2}}=1  \quad (\text{Lemma~\ref{3lem equiv}}).\]
We let $M_f=\overline{\Span}\{f_n:n\ge 1\}$ and we show that $M_f$ is a hypercyclic subspace.

We first deduce from (ii) that if $\psi(l)=(n,j,N)$ and thus $x_{l}=\sum_{i=0}^{K}S_{k(n_l)+iN}y_j$, then
for any $i\ge 0$, if $k(n_l)+iN\in[k(n_l),k(n_l+1)+N[$, we have
\begin{align*}
&p_{j}(T^{k(n_l)+iN}x_l-y_j)\\
&\quad\le p_j(T^{k(n_l)+iN}S_{k(n_l)+iN}y_j-y_j)+
p_j\Big(\sum_{0\le i'\le K, i'\ne i}T^{k(n_l)+iN}S_{k(n_l)+i'N}y_j\Big)\\
&\quad\le \frac{1}{2^{l+j+1}}+ \delta(j,N)
\end{align*}
where $\delta(j,N)=\sup\{p_j(\sum_{i'\in F, i'\ne i}T^{k+iN}S_{k+i'N}y_j): k\ge 0,\ F\subset \Z_+\ \text{finite}\}$ and thanks to Conditions~1. and 2. of the Frequent Hypercyclicity Criterion, we know that for any $j\ge 1$,
\begin{equation}
\delta(j,N)\xrightarrow[N\rightarrow \infty]{} 0.
\label{6lim0}
\end{equation}

Therefore, for any $n,j,N\ge 1$, if $\psi(l)=(n,j,N)$ and $k(n_l)+iN\in[k(n_l),k(n_l+1)+N[$, we have
\begin{equation}
p_j(T^{k(n_l)+iN}(f_n-u_n)-y_j)\le \frac{1}{2^{l+j}}+ \delta(j,N)\label{6hered1fhc}
\end{equation}
because 
\begin{align*}
p_j(T^{k(n_l)+iN}(f_n-u_n)-y_j)&\le \sum_{l'\in A_n\backslash\{l\}}p_{l'+j}(T^{k(n_l)+iN}x_{l'})
+ p_{j}(T^{k(n_l)+iN}x_l-y_j)\\
&\le \sum_{l'\ne l} \frac{1}{2^{l'+l+n+j}} + \frac{1}{2^{l+j+1}}+ \delta(j,N) \quad \text{by (iv)}\\
&\le \frac{1}{2^{l+j}}+ \delta(j,N).
\end{align*}
We also remark that for any $l\ge 1$ with $\psi(l)=(n,j,N)$, for any $m\ne n$, for any $k(n_l)\le k< k(n_l+1)+N$, we have
\begin{equation}
p_j(T^k(f_m-u_m))< \frac{1}{2^{l+j+m}}\label{6hered2}
\end{equation}
because
\begin{align*}
p_j(T^k(f_m-u_m))&\le \sum_{l'\in A_m} p_{j}(T^k x_{l'})\\
&\le \sum_{l'\ne l} p_{l'+j}(T^k x_{l'}) < \sum_{l'=1}^{\infty} \frac{1}{2^{l+l'+j+m}} \quad\text{by (iv)}\\
&\le \frac{1}{2^{l+j+m}}.
\end{align*} 

Let $x=\sum_{m\ge1}\alpha_m f_m\in M_f\backslash\{0\}$. There exists $n\ge 1$ such that $\alpha_n\ne 0$ and without loss of generality, we can suppose that $\alpha_n=1$. Moreover, we know that the sequence $(\alpha_m)_m$ is bounded by some constant $K$ (Remark~\ref{rembdd}). Therefore, for any $j,N\ge 1$, if $\psi(l)=(n,j,N)$ and if $k(n_l)+iN\in[k(n_l),k(n_l+1)+N[$, we have
\begin{align*}
&p_j\Big(T^{k(n_l)+iN}\Big(\sum_m \alpha_m f_m\Big)-y_j\Big)\\
&\quad\le  p_j\Big(T^{k(n_l)+iN}\Big(\sum_m \alpha_m u_m\Big)\Big)
+ p_j(T^{k(n_l)+iN}(f_{n}-u_{n})-y_j)\\
&\hspace*{5.05cm}+ \sum_{m\ne n} K p_j(T^{k(n_l)+iN}(f_m-u_m))\\
&\quad\le p_j\Big(T^{k(n_l)+iN}\Big(\sum_m \alpha_m u_m\Big)\Big)+ \frac{1}{2^{l+j}}+ \delta(j,N)+ \sum_{m\ne n} \frac{K}{2^{l+j+m}}
\quad\text{by \eqref{6hered1fhc} and \eqref{6hered2}}\\
&\quad\le p_j\Big(T^{k(n_l)+iN}\Big(\sum_m \alpha_m u_m\Big)\Big)+ \frac{1}{2^{l+j}}+ \delta(j,N)+ \frac{K}{2^{l+j}}.
\end{align*}

By continuity of $T$, for any $j,N\ge 1$, there exists $m_{j,N}\ge 1$ and $C>0$ such that for any $k\le N$
\[p_j(T^kx)\le Cp_{m_{j,N}}(x) \quad \text{for any\ }x\in X.\]
Since $\sum_m \alpha_m u_m\in M_0$, we know that for any $j,N\ge 1$, there exists an increasing sequence $(k_l)_{l\ge 1}$ such that $k_l\in[k(n_l),k(n_l+1)[$ and $p_{m_{j,N}}\big(T^{k_l}\big(\sum_m \alpha_m u_m\big)\big)$ tends to $0$ as $l$ tends to infinity. Let $i_l\ge 1$ such that $k_l\in[k(n_l)+(i_l-1)N,k(n_l)+i_lN[$. We have $k(n_l)+i_lN\in[k(n_l),k(n_l+1)+N[$ and thus for any $l\ge 1$, if $\psi(l)=(n,j,N)$,
\begin{align*}
&p_j\Big(T^{k(n_l)+i_lN}\Big(\sum_m \alpha_m f_m\Big)-y_j\Big)\\
&\quad \quad \le p_j\Big(T^{k(n_l)+i_lN}\Big(\sum_m \alpha_m u_m\Big)\Big)+ \frac{K+1}{2^{l+j}}+ \delta(j,N)\\
&\quad \quad \le C p_{m_{j,N}}\Big(T^{k_l}\Big(\sum_m \alpha_m u_m\Big)\Big) + \frac{K+1}{2^{l+j}}+\delta(j,N)\xrightarrow[l\rightarrow \infty]{} \delta(j,N).
\end{align*}
We deduce that for any $x\in M_f\backslash\{0\}$, any $\varepsilon>0$, any $j,N\ge 1$, there exist $l\ge 1$ and $i\ge 1$ such that
\[p_j\Big(T^{k(n_l)+iN}x-y_j\Big)\le \delta(j,N)+\varepsilon.\]
In view of \eqref{6lim0}, we therefore conclude that $x$ is hypercyclic.
\end{proof}

We can generalize this result to hypercyclic subspaces of type~$1$ as follows.
\begin{theorem}
Let $X$ be an infinite-dimensional separable Fr\'{e}chet space and $T\in L(X)$.
If $T$ satisfies the Frequent Hypercyclicity Criterion, then the following assertions are equivalent:
\begin{enumerate}[\upshape (1)]
\item the operator $T$ possesses a hypercyclic subspace of type~$1$;
\item there exist an infinite-dimensional closed subspace $M\subset X$ and a continuous seminorm $p$ on $X$
such that 
\begin{itemize}
\item $M\cap \ker p$ is of infinite codimension in $M$,
\item for any continuous seminorm $q$ on $X$, we have
\[\liminf_{k\to \infty} q(T^kx)<\infty \quad\text{for any }x\in M;\]
\end{itemize}
\item there exist an infinite-dimensional closed subspace $M_0\subset X$, an increasing sequence of integers $k(n)$ and a continuous seminorm $p$ on $X$ such that 
\begin{itemize}
\item $M_0\cap \ker p$ is of infinite codimension in $M_0$,
\item for any continuous seminorm $q$ on $X$, we have
\[\lim_{n\to \infty} \inf_{k(n-1)\le k\le k(n)} q(T^kx)=0 \quad \text{for any }x\in M_0.\]
\end{itemize}
\end{enumerate}
\end{theorem}

In conclusion, we have obtained a characterization of sequences of operators $(T_n)$ possessing hereditarily hypercyclic subspaces of type~$1$ under the assumption that $(T_n)$ satisfies the Hypercyclicity Criterion and a characterization of operators with hypercyclic subspaces of type~$1$ in the case of operators satisfying the Frequent Hypercyclicity Criterion. It would be interesting to know whether there exists an operator $T$ on some Fr\'{e}chet space such that $T$ possesses a hypercyclic subspace but $T$ does not possess hereditarily hypercyclic subspaces and if such an operator can be weakly mixing.

\end{document}